\documentclass[11pt]{amsart} 
\usepackage[a4paper, margin=3cm]{geometry}
\usepackage{enumerate,amsmath,amssymb,amsthm,bbm,url}

\theoremstyle{plain}
\newtheorem{Theorem}{Theorem}
\newtheorem{Lemma}{Lemma}
\newtheorem{Proposition}{Proposition}
\theoremstyle{definition}
\newtheorem{Definition}{Definition}
\theoremstyle{remark}
\newtheorem{Remark}{Remark}

\newcommand{\set}[1]{\left\{#1\right\}}
\newcommand{\norm}[1]{\left\lVert#1\right\rVert}
\newcommand{\ind}[1]{\mathbbm{1}_{#1}}
\newcommand{\abs}[1]{\left\vert#1\right\vert}

\DeclareMathOperator{\Log}{Log}

\DeclareMathOperator{\sign}{sign}
\newcommand{\R}{\mathbb{R}}
\newcommand{\prob}{\mathrm{P}}
\newcommand{\ex}[2][]{\mathrm{E}_{#1}\left[\,#2\,\right]}
\newcommand{\eps}{\varepsilon}
\begin{document}

\author{Yuri Kondratiev \and Yuliya Mishura \and Georgiy Shevchenko}
\title[Limit theorems for additive functionals of CTRW]{Limit theorems for additive functionals\\ of continuous time random walks}

\begin{abstract}
For a  continuous-time random walk $X=\set{X_t,t\ge 0}$ (in general non-Markov), we study the asymptotic behavior, as $t\rightarrow \infty$, of the normalized additive functional $c_t\int_0^{t} f(X_s)ds$, $t\ge 0$. Similarly to the Markov situation, assuming that the distribution of jumps of $X$ belongs to the domain of attraction to $\alpha$-stable law with $\alpha>1$, we establish the convergence  to the local time at zero of an $\alpha$-stable L\'evy motion. We further study a situation where $X$ is delayed by a random environment given by the Poisson shot-noise potential:
$\Lambda(x,\gamma)= e^{-\sum_{y\in \gamma} \phi(x-y)},$
where $\phi\colon\mathbb R\to [0,\infty)$ is a bounded function decaying sufficiently fast, and $\gamma$ is a homogeneous Poisson point process, independent of $X$. We find that in this case the weak limit has both ``quenched'' component depending on $\Lambda$, and a component, where $\Lambda$ is ``averaged''. 
\end{abstract}


\maketitle

\section{Introduction}

An evolution of continuous-time random walk (CTRW) $X=\set{X_t,t\ge 0}$ is  described by a sequence of times between consecutive jumps of the process, which are assumed to be independent identically distributed (iid) positive random variables $\theta_n$, $n\ge 1$, and by a sequence of iid sizes of jumps $\xi_n$, $n\ge 1$; the two sequences are assumed to be independent. When the distribution of $\theta_n$ is exponential, CTRW is nothing but a compound Poisson process. Otherwise, CTRW is in general not a Markov process, so may be considered as a non-Markovian generalization of a compound Poisson process. 

It is handful to represent the CTRW $X$ in the form
\begin{equation}\label{eq:ctrw}
X_t = \sum_{k=1}^{N_t} \xi_k,
\end{equation}
where $N_t = \max\set{k\ge 0: \sum_{i=1}^k \theta_i \le t}$ is the number of jumps up to time $t$. (Throughout the paper we use the convention that $\sum_{k=1}^{0}=0.$)

Consider a function $f\colon \R\to \R$. We are interested in the asymptotic behavior, as $t\rightarrow \infty$, of the additive functional $\int_0^t f(X_s)ds$, normalized by a suitable factor.

When $X$ is a discrete- or continuous-time ergodic Markov process having an invariant probability measure $\nu$, additive functionals of the form $A_t = \sum_{i=1}^t f(X_i)$ (respectively $A_t=  \int_0^t f(X_s)ds$) with $f\in L^1(\nu)$ satisfy strong law of large numbers: $A_t/t \to \nu(f):=\int f(x)\nu(dx)$, $t\to\infty$, almost surely, and, under some additional assumptions, a central limit theorem: $\bigl(A_t - \nu(f)t\bigr)/\sqrt{t}\overset{d}{\longrightarrow}{N(0,\sigma^2_f)}, t\to\infty$, with some variance $\sigma^2_f$ (see e.g.\ \cite[Chapter 2]{eberle}).

The situation is very different when $X$  does not have an invariant probability measure, in particular, when it is a random walk. In this case, under suitable normalization, additive functionals converge to a local time of some $\alpha$-stable L\'evy motion multiplied by the integral of $f$ (or, in the case of lattice random walk, by the sum of its values at the lattice points) (see \cite{Borodin,Jeganathan}). It is also worth to mention works \cite{kartashov-kulik,kulik}, where a general result on convergence of additive functionals of Markov processes is proved, and \cite{jacod}, which studies convergence to local times and associated central limit theorems for additive functionals of diffusions. 

There are also results in the non-Markovian case. Most notably, \cite{Jeganathan} studies cumulative sums $S_k = \sum_{i=1}^k X_i$ of some long-memory stationary sequences $X$ of moving-average time, and establishes convergence of normalized additive functionals to the local time of fractional Brownian motion or, in a heavy-tail situation, of a fractional $\alpha$-stable process (it is also worth mentioning that this article establishes some of the strongest results for the Markovian situation as well).

To the best of our knowledge, the asymptotic behavior of additive functionals for a CTRW has not been studied yet in the literature. We are focusing on the case where the times between jumps are integrable. In this case, despite the corresponding CTRW is possibly a non-Markovian, the results are similar to the Markovian case. The reason is that the process $N_t$ grows approximately linearly,  thanks to the law of large numbers; the corresponding results are contained in Section 3. In Section 4, we consider a quite different situation where the process $X$ is delayed by some environment $\Lambda$. We first study the case of non-random $\Lambda$, and prove a corresponding limit theorem. Further we look at a random environment given by the Poisson shot noise potential
$$
\Lambda(x,\gamma)= e^{-\sum_{y\in \gamma} \phi(x-y)},
$$
where $\gamma$ is a homogeneous Poisson configuration, and $\phi\colon\R\to [0,\infty)$ is bounded and integrable. We establish a limit theorem for this case as well. The  convergence we show is ``quenched'' in the sense that we have a weak convergence to a limit depending on $\gamma$ for almost all configurations $\gamma$. Another interesting feature is that the limit, besides the aforementioned ``quenched'' component, contains a component, where $\Lambda$ is ``averaged''. 

The remaining structure of the article is following: Section 2 contains some preliminary information on domains of attraction and stable variables, and proofs, which are rather technical, are postponed to Appendix. 

\section{Preliminaries}
For any random variable $X$, we denote by $\varphi_{X}(\lambda)=\ex{e^{i\lambda X}}$ its characteristic function. If $X$ has absolutely continuous distribution, $f_X$ denotes its density. Throughout the proofs, $C$ is a generic constant (possibly random), the value of which is not important and may change between lines. To emphasize dependence on some variables, we put them in subscripts: $C_p, C_k$ etc. The symbols $\overset{d}{\longrightarrow}$ and $\overset{fdd}{\longrightarrow}$ designate the convergence in law and the convergence of finite dimensional distributions, respectively.
\subsection{Domains of attraction}

Consider the basic definitions concerning the random variables $\{\xi_n,n\ge 1\}$, for details see \cite[Chapter XVII]{feller2} and \cite{zolotarev}.

\begin{Definition}
A random variable $\xi$ is said to have a stable distribution with index $\alpha \in (1,2]$ if its characteristic function has the form
\begin{equation*}
\varphi_{\xi}(x)=\exp\left\{iax-c|x|^\alpha \omega(x,\alpha,\beta)\right\}, \quad c>0,\quad a\in R,~\alpha\in(1,2],
\end{equation*}
where $\omega(x,\alpha,\beta)=1+i\beta\sign x\tan\left(\frac{\pi\alpha}{2}\right)$; $c>0$ is called the scale parameter, $\beta\in[-1,1]$ is called the skewness parameter, $a\in \R$ is the expected value. 
\end{Definition}
\begin{Definition}\label{def-stable}
The distribution $\mathcal L$ is said to belong to the domain of attraction to stable law with index $\alpha\in (1,2]$ if there exist some sequence $a_n\in \R$ and a slowly varying function $L$ such that the normalized sums $$\frac{\xi_1+\dots +\xi_n}{L(n)n^{1/\alpha}} -a_n$$ of iid random variables $\{\xi_n,n\ge 1\}$ with distribution $\mathcal L$ converge, as $n\to\infty$, to a stable distribution with index $\alpha$.
\end{Definition}
\begin{Definition}If in Definition \ref{def-stable} $L(n)=\sigma$ for some constant $\sigma>0$, we say that $\mathcal L$ belongs to the domain of normal attraction to stable law with index $\alpha\in(1,2]$.\end{Definition}

A distribution $\mathcal L$ belongs to a domain of attraction of a stable law with index $\alpha$ if its characteristic function admits in some neighborhood of $0$ an expansion of the form
\begin{equation}\label{character2}
\varphi_{\mathcal L}(x) = \exp\bigl\{iax-  h(|x|)|x|^\alpha\omega(x,\alpha,\beta)\bigr\},
\end{equation}
where $h(x)$ is a function slowly varying at 0; it belongs to the domain of normal attraction to a stable law with index $\alpha\in(1,2]$ if $h(x)\to c>0$, $x \to 0$. The relation between $h$ and $L$ is as follows:
$$\frac1{L(n)n^{1/\alpha}}=\inf\left\{x>0: x^\alpha h(x)=\frac{1}{\alpha n} \right\}, n\ge 1.$$

\section{Asymptotic behavior of additive functionals for CTRW}

In this section we study asymptotic behavior of additive functionals of the form $\int_0^t f(X_t)dt$  for CTRW  $X$ given by \eqref{eq:ctrw}. We will need several assumptions concerning the distribution of jumps $\xi_n$ and times between them $\theta_n$ as well as function $f$. 

\begin{itemize}
\item[\textbf{A1.}] We will assume that the jump sizes $\xi_n$, $n\ge 1$, are centered and their distribution  belongs to the domain of attraction to $\alpha$-stable law with $\alpha\in(1,2]$. In this case (see e.g. \cite[Proposition 3.4]{resnick}) there is also a functional convergence
$$
\Biggl\{\frac{1}{L(n)n^{1/\alpha}}\sum_{k=1}^{[nt]} \xi_k,t\ge 0\Biggr\} \overset{d}{\longrightarrow} \set{Z_\alpha(t), t\ge 0}
$$
towards an $\alpha$-stable L\'evy motion $Z_\alpha(t)$. 
\item[\textbf{A2.}] The assumptions on function $f$ come from \cite{Jeganathan} and are accompanied by additional assumptions on the distribution of $\xi_1$, namely, we assume that
\begin{itemize}
\item[(i)] either $f\in L^1(\R)\cap L^\infty(\R)$ and  the distribution of $\xi_1$ has a nonzero absolutely continuous component,
\item[(ii)] or $f\in L^1(\R)\cap L^2(\R)$ and the characteristic function $\varphi_{\xi_1}$  of jump sizes is integrable to some power $p>0$: 
$$
\int_{-\infty}^\infty \abs{\varphi_{\xi_1}(t)}^p dt<\infty. 
$$
\end{itemize}
\item[\textbf{A3.}] Concerning the times $\theta_n$ between jumps, we will assume that they are integrable:
$$
\ex{\theta_1} = \mu.
$$
\end{itemize}

Denote $c_t = L(t)t^{1/\alpha-1}$, $S_n = \xi_1+\dots + \xi_n$.  We will use the following result, which is an adaptation of Theorem 3 from \cite{Jeganathan} for the case $x=0$, $\beta_n = y_n$, $c_0=1$, $c_j=0$, $j\ge 1$ (in the terms of \cite{Jeganathan}).

\begin{Theorem}[\cite{Jeganathan}]\label{thm:jeganathan}
Under assumptions {\rm A1--A2}, the finite-dimensional distributions of the process $$
c_n \sum_{k=1}^{[nt]} f(S_k),\quad  t\ge 0,
$$
converge to those of 
$$
\int_{-\infty}^\infty f(x)dx\cdot  \ell_\alpha(t,0),\quad t\ge 0,
$$
where $\ell_\alpha(t,0)$ is the symmetric local time at zero of the $\alpha$-stable L\'evy motion $Z_\alpha$ on $[0,t]$. 
\end{Theorem}

Now we establish a similar result for the CTRW $X$ defined by \eqref{eq:ctrw}.
\begin{Theorem}\label{thm:main}
Let $X$ be given by \eqref{eq:ctrw}. Under assumptions {\rm A1--A3}, the finite-dimensional distributions of the process
$$
c_t\int_0^{tu} f(X_s)ds,\quad  u\ge 0,
$$
converge as $t\to+\infty$ to those of 
$$
\mu^{1/\alpha}\int_{-\infty}^\infty f(x)dx\cdot  \ell_\alpha(u,0),\quad u\ge 0.
$$
\end{Theorem}

\begin{Remark}
Using the results of \cite{Borodin}, it is possible to replace the additional assumptions from A2 on the distribution of $\xi$ by the requirement to be non-lattice. However, in this case $f$ should have a compactly supported Fourier transform, which is a very restrictive requirement. 
\end{Remark}

\begin{Remark}
The results of \cite{Borodin} can be used to handle the lattice case. Namely, let A1 and A3 hold, but A2 is replaced by the assumption that $\xi_1(\Omega) = \set{a+b\mathbb{Z}}$ for some $a\in \mathbb{R}, b>0$, and $\sum_{n=-\infty}^\infty |f(a+bn)|<\infty$.  Then 
$$
\set{c_t\int_0^{tu} f(X_s)ds,u\ge 0} \overset{fdd}{\longrightarrow}  \set{\mu^{1/\alpha} b \sum_{n=-\infty}^\infty f(a+bn) \cdot  \ell_\alpha(u,0),u\ge 0}, t\to+\infty.
$$
\end{Remark}

\section{CTRW in a random environment}

\subsection{CTRW with location-dependent intensity of jumps}
Consider now the situation that the time between jumps depends on the current location of the random walker: the intensity of jumps from a location $x$ is $\Lambda(x)>0$. 
In the Markovian case, the corresponding evolution is a pure jump process with the generator
$$(\mathcal A \psi)(x) = \Lambda(x)\int_{\R} \bigl(\psi(x-y)-\psi(x)\bigr) F_{\xi_1}(dy).$$
The consecutive locations visited by the random walker $X$ are, as before, $S_1 = \xi_1, S_2 = \xi_1+\xi_2, \dots, S_n = \sum_{k=1}^n\xi_k, \dots$  The time spent in the $n$th location is an exponential random variable with parameter $\lambda(S_n)$, which also can be written as $\theta_n/\lambda(S_n)$, where $\theta_n$ is an exponential random variable with parameter 1. In view of independence of times between jumps, the random variables $\theta_n$, $n\ge 1$, are independent,
so the evolution can be written in the form
\begin{equation}\label{eq:ctrw-lambda}
X_t = \sum_{k=1}^{N_t} \xi_k,
\end{equation}
where $N_t = \max\set{k\ge 0: \sum_{i=1}^k \theta_i/\Lambda(S_i) \le t}$. To construct a non-Markovian counterpart of this dynamic, we now drop the requirement that the variables $\theta_n$, $n\ge 1$, have exponential distribution. So in the rest of this section $X$ will be given by \eqref{eq:ctrw-lambda} with iid jumps $\xi_n$, $n\ge 1$, and iid variables $\theta_n$, $n\ge 1$, which are also independent of $\xi$.

In this section we will need stronger assumptions than in the previous one. Namely, we will assume that the jump sizes are from the \textit{normal} domain of attraction of $\alpha$-stable law. Moreover, since the case $\alpha = 2$ is very different technically, we will consider in this section only non-Gaussian case $\alpha\in (1,2)$. We will also need stronger assumptions on the distribution of jumps.

\begin{itemize}
\item[\textbf{B1.}] The jump sizes $\xi_n$, $n\ge 1$, are centered and their distribution  belongs to the normal domain of attraction to $\alpha$-stable law with $\alpha\in(1,2)$, i.e.\ $L(n) = \sigma>0$ in Definition~\ref{def-stable}. In this case (see \cite[Proposition 3.4]{resnick}) there is  a functional convergence
$$
\Biggl\{\frac{1}{\sigma  n^{1/\alpha}}\sum_{k=1}^{[nt]} \xi_k,t\ge 0\Biggr\} \overset{d}{\longrightarrow} \set{Z_\alpha(t), t\ge 0}
$$
towards an $\alpha$-stable L\'evy motion $Z_\alpha(t)$. 

\item[{\textbf{B2}.}] The distribution of $\xi_n$ is absolutely continuous with 
$$
\int_{-\infty}^\infty x^2 \abs{f_{\xi_1}(x) - f_{Z_\alpha}(x)} dx<\infty. 
$$
\end{itemize}

Concerning the jump intensity $\Lambda$ we will assume sub-polynomial growth and existence of Cezaro averages for its inverse. 

\begin{itemize}
\item [{\rm\textbf{B3.}}]
For any $\delta>0$, $\sup_{|x|\le n}\Lambda(x)^{-1} = o(n^\delta), n\to \infty$. 
\item [{\rm\textbf{B4.}}] There exists $\overline{\Lambda^{-1}} > 0$ such that for some $r>\alpha$,
$$
\sup_{|x|\le t^r} \abs{\frac{1}{t} \int_x^{x+t} \Lambda(y)^{-1}dy - \overline{\Lambda^{-1}}}\to 0, t \to +\infty.
$$
\end{itemize}

We start by examining the properties of the sums $\sum_{i=1}^n \theta_i/\Lambda(S_i)$ and the process $N_t$.

\begin{Proposition}\label{prop:N_t}
Under the assumptions {\rm A3, B1--B4}, 
$$
\frac{1}{n}\sum_{i=1}^n \frac{\theta_i}{\Lambda(S_i)}\overset{\prob}{\longrightarrow} \mu \overline{\Lambda^{-1}}, n\to \infty,
$$
and 
$$
\frac{N_t}{t} \overset{\prob}{\longrightarrow} \frac{1}{\mu \overline{\Lambda^{-1}}}, t\to\infty. 
$$
\end{Proposition}

Finally we turn to asymptotics of the additive functional. 

\begin{Theorem}\label{thm:main-lambda}
Let $X$ be given by \eqref{eq:ctrw-lambda}. Under assumptions {\rm A2--A3} on $f(x) = \frac{g(x)}{\Lambda(x)}$ and {\rm B1--B4}, the finite-dimensional distributions of the process
$$
\sigma t^{1/\alpha-1}\int_0^{tu}g(X_s)ds,\quad  u\ge 0,
$$
converge as $t\to+\infty$ to those of 
$$
\mu^{1/\alpha} \cdot\bigl( \overline{\Lambda^{-1}}\bigr)^{1/\alpha-1}\cdot \int_{-\infty}^\infty \frac{g(x)}{\Lambda(x)}dx\cdot  \ell_\alpha(u,0),\quad u\ge 0.
$$
\end{Theorem}

\subsection{CTRW in a Poisson shot-noise potential environment}

The conclusion of Theorem~\ref{thm:main-lambda} is also true for a random $\Lambda$ independent of $X$ provided that $\Lambda$ satisfies B3--B4 almost surely, and $g/\Lambda$ satisfies one of the assumption A2(i) or A2(ii) almost surely. 

Of particular interest is a random $\Lambda$ of the special form, a so-called Poisson shot-noise potential:
\begin{equation}\label{eq:pois-lambda}
\Lambda(x,\gamma)= e^{-\sum_{y\in \gamma} \phi(x-y)}  =: e^{-E_\phi(x,\gamma)},
\end{equation}
where $\phi\colon\R\to [0,\infty), \gamma$ is a homogeneous Poisson configuration, i.e.\ a point process such that for any Borel set $A\subset \R$ having finite Lebesgue measure $\lambda(A)$, the number of points of $\gamma$ in $A$, $\abs{\gamma \cap A}$, has a Poisson distribution with parameter $\lambda(A)$. A sufficient condition for $\Lambda$ to be well defined for almost all $x\in\R^d$ is that $\phi\in L^1(\R)$.
To ensure the assumptions B3 and B4, we will need a stronger assumption. 

\begin{itemize}
\item[\textbf{C1.}] $\phi \in C(\mathbb R)$ and there exist some $C,\beta>0$ such that $|\phi(x)|\le \frac{C}{1+|x|^{1+\beta}}$. 
\end{itemize}
Under this assumption, 
$$
\ex{\Lambda(x,\gamma)^{-a} } = \exp\left\{ \int_{-\infty}^\infty \bigl(e^{a\phi(y)}-1\bigr)dy \right\}
$$
for any $a\in \R$ and
\begin{equation}\label{eq:E-phi}
\sup_{|x|\le n}\abs{E_\phi(x,\gamma)} = O\left(\frac{\log |n|}{\log \log |n| }\right), n\to\infty,
\end{equation}
a.s. (see \cite{carmona}). 

\begin{Proposition}\label{prop:b3b4}
Under the assumption {\rm C1}, for any $\delta>0$, 
\begin{equation} \label{eq:B3}
\sup_{|x|\le n}\Lambda(x,\gamma)^{-1} = o(n^\delta), n\to \infty,
\end{equation}
 a.s. and for any $r>1$, 
$$
\sup_{|x|\le t^r} \abs{\frac{1}{t} \int_x^{x+t} \Lambda(y,\gamma)^{-1}dy - \ex{\Lambda(0,\gamma)^{-1}}}\to  0, t \to +\infty,
$$
almost surely.
\end{Proposition}

We are now in the position to prove the main result of this section. To ensure A2 for the function $g/\Lambda$, we impose suitable assumptions on $g$. 

\begin{itemize}
\item[\textbf{C2.}]Either $g\in L^1(\R)\cap L^2(\R)$, $e^{2\phi}-1\in L^1(\R)$ and the characteristic function $\varphi_{\xi_1}$  of jump sizes is integrable to some power $p>0$, or 
$g\in L^1(\R)$ and there exist some $C,\eps>0$ such that $|g(x)|\le C(1+|x|^\eps)^{-1}$ for all $x\in \R$. 
\end{itemize}
\begin{Theorem}\label{thm:main-poisson}
Let $X$ be given by \eqref{eq:ctrw-lambda} and $\Lambda$ be given by \eqref{eq:pois-lambda} with $\gamma$ independent of $X$. Under assumptions {\rm A3, B1, B2, C1, C2}, the finite-dimensional distributions of the process
$$
\sigma t^{1/\alpha-1}\int_0^{tu}g(X_s)ds,\quad  u\ge 0,
$$
converge as $t\to+\infty$ to those of 
$$
\mu^{1/\alpha}\cdot\exp\set{\Big(\frac1\alpha-1\Big)\int_{-\infty}^\infty \bigl(e^{\phi(y)}-1\bigr)dy}\cdot \int_{-\infty}^\infty \frac{g(x)}{\Lambda(x,\gamma)}dx\cdot  \ell_\alpha(u,0),\quad u\ge 0,
$$
with $\ell_\alpha$ independent of $\gamma$. 
\end{Theorem}

\subsection*{Acknowledgments}
Georgiy Shevchenko is grateful to Aleksei  Kulik, Professor  of Wroclaw University of Science and Technology, for fruitful discussions that led to the proof of Lemma \ref{lem:sum1/lambda}.

\appendix
\section{Proofs and auxiliary results}

\begin{proof}[Proof of Theorem \ref{thm:main}]
For simplicity, we show marginal convergence for $u=1$; for arbitrary finite-dimensional distributions the proof is the same, just heavier in terms of notation.

Denote $\tau_n = \sum_{k=1}^{n}\theta_k$, $n\ge 0$, and write 
\begin{equation}\label{decompos}
c_t\int_0^t f(X_s)ds = c_t\sum_{k=1}^{N_t} \theta_k f(S_{k-1}) + c_t\big(t-\tau_{N_t}\big) f(S_{N_t}).
\end{equation}
By the strong law of large numbers, $N_t \sim  t/\mu$, $t\to\infty$.
Therefore, since $c_t$ is regularly varying at infinity of index $1/\alpha-1$, 
\begin{equation}\label{eq:ct}
c_t\sim \mu^{1/\alpha -1}c_{N_t}, t\to \infty,
\end{equation}
 a.s.  
Thus, thanks to Slutsky's lemma, we need to study the asymptotics of the normalized sums 
$$
\zeta_n = c_n\sum_{k=1}^{n} \theta_k f(S_{k-1})\;\text{as}\;n\rightarrow \infty
$$
(the remainder $c_t\big(t-\tau_{N_t}\big) f(S_{N_t})$ will be handled later).

\textit{Step 1}. Let us first consider the case of a bounded $f$. Thanks to independence of $\xi$ and $\theta$, we can write
\begin{equation*}\begin{gathered}
\varphi_{\zeta_n}(\lambda) = \ex{\ex{\set{i\lambda c_n \sum_{k=1}^n\theta_k x_k}}\Bigg|_{x_k=f(S_{k-1}) ,k=1,\dots,n}} \\
= \ex{\prod_{k=1}^n \varphi_{\theta_1}\bigl(\lambda c_n f(S_{k-1})\bigr)}
= \ex{\exp\set{\sum_{k=1}^n \Log\varphi_{\theta_1}\bigl(\lambda c_n f(S_{k-1})\bigr)}},
\end{gathered}\end{equation*}
where $\Log$ denotes the branch of the natural logarithm such that $\Log z\in (-\pi,\pi]$ for all $z\in \mathbb{C}\setminus\{0\}$. 
From assumption A3  we have
$$
\varphi_{\theta_1}(t) = 1 + i\mu t + o(t), t\to 0.
$$
Since also
$$
\Log(1+x) - x = o(x), x\to 0,
$$
we get 
$$
r(t):= \Log\varphi_{\theta_1}(t) - i\mu t = o(t), t\to 0. 
$$
Now
\begin{equation}\label{eq:zetachf}\begin{gathered}
\varphi_{\zeta_n}(\lambda) =  \ex{\exp\set{\sum_{k=1}^n \bigl( i\mu \lambda c_n f(S_{k-1}) + R_{k,n}\bigr)}}\\
 =  \ex{\exp\set{i\mu \lambda c_n\sum_{k=1}^n   f(S_{k-1}) + R_{n}}},
\end{gathered}\end{equation}
where $R_{k,n} = r\bigl(\lambda c_n f( S_{k-1})\bigr)$, $R_n = \sum_{k=1}^n R_{k,n}$. By Theorem \ref{thm:jeganathan}, 
$$
c_n\sum_{k=1}^n   f(S_{k-1}) \overset{d}{\longrightarrow} \int_{-\infty}^\infty f(x)dx\cdot  \ell_\alpha(1,0), n\to\infty.  
$$
Since the absolute value of the expression inside the expectation in \eqref{eq:zetachf} is bounded by 1, 
we just need to show that $R_n\to 0$, $n\to\infty$, in probability. To this end, fix arbitrary $\eps>0$ and let $\delta>0$ be such that $\abs{r(t)}< \eps\abs{t}$ whenever $\abs{t}<\delta$. Since $f$ is bounded, $\lambda c_n |f(x)|\le \delta$ for all $x\in\R$ and all $n$ large enough. Then we can write
\begin{gather*}
\abs{R_n}\le \sum_{k=1}^n \abs{R_{k,n}} 
\le \eps \abs{\lambda} c_n \sum_{k=1}^n  \abs{f(S_{k-1})}.
\end{gather*}
By Theorem 1, 
\begin{equation}\label{eq:absconv}
c_n\sum_{k=1}^n  \abs{f(S_{k-1})} \overset{d}{\longrightarrow} \int_{-\infty}^\infty |f(x)|dx\cdot  \ell_\alpha(1,0), n\to\infty, 
\end{equation}
so for any $\eta>0$,
$$
\limsup_{n\to\infty}\prob \left(R_{n}\ge  \eta\right)\le \prob\left(\int_{-\infty}^\infty |f(x)|dx\cdot  \ell_\alpha(1,0)\ge \frac{\eta}{\eps|\mu\lambda|}\right).
$$
Letting $\eps\to 0+$, we arrive at $\limsup_{n\to\infty}\prob \left(R_n\ge  \eta\right) = 0$, which gives the desired convergence in probability. 

Consequently, from the L\'evy theorem we get 
\begin{equation}\label{eq:discreteconv}
\zeta_n \overset{d}{\longrightarrow} \mu \int_{-\infty}^\infty f(x)dx\cdot  \ell_\alpha(1,0), n\to\infty. 
\end{equation}

\textit{Step 2}. Now let $f$ be unbounded. We are going to apply \cite[Theorem~3.2]{billingsley}. As we have just shown,  for any $m \ge 1$,
$$
\zeta_{n}^m:=c_n\sum_{k=1}^{n} \theta_k f(S_{k-1}) \ind{\abs{f(S_{k-1})}\le m} \overset{d}{\longrightarrow} \zeta^m:=\mu \int_{-\infty}^\infty f(x)\ind{|f(x)|\le m}dx\cdot  \ell_\alpha(1,0),\ n\to\infty.
$$
It is also clear that
$$
\zeta^m \overset{d}{\longrightarrow} \mu\int_{-\infty}^\infty f(x)dx\cdot  \ell_\alpha(1,0), m\to\infty. 
$$
So it remains to deal with 
$$
\abs{\zeta_n - \zeta_{n}^m} = c_n\abs{\sum_{k=1}^{n} \theta_k f(S_{k-1}) \ind{\abs{f(S_{k-1})}> m}}\le c_n\sum_{k=1}^{n} \theta_k \abs{f(S_{k-1})} \ind{\abs{f(S_{k-1})}> m}.
$$
Denote $f_m(x) = \abs{f(x)} \ind{\abs{f(x)}> m}$. For any $\eps>0$, owing to independence of $\xi$ and $\theta$, we can write
\begin{equation}\label{eq:probab>eps}\begin{gathered}
\prob\bigl(\abs{\zeta_n - \zeta_{n}^m}>\eps \bigr) \le \ex{\prob\left(c_n\sum_{k=1}^{n} \theta_k x_k>\eps\right)\Bigg|_{x_k=f_m(S_{k-1}) ,k=1,\dots,n}}.
\end{gathered}\end{equation}
Thanks to the Markov inequality,
$$\prob\biggl(c_n\sum_{k=1}^{n} \theta_k x_k>\eps\biggr)\le \frac{c_n}{\eps}\ex{\sum_{k=1}^{n} \theta_k x_k} = \frac{c_n\mu }{\eps}\sum_{k=1}^{n}  x_k,$$
so 
$$
\prob\biggl(c_n\sum_{k=1}^{n} \theta_k x_k>\eps\biggr)\Bigg|_{x_k=f_m(S_{k-1}),k=1,\dots,n}\le \frac{c_n\mu }{\eps}\sum_{k=1}^{n}f_m(S_{k-1}).
$$
By Theorem~\ref{thm:jeganathan}, 
$$
c_n\sum_{k=1}^{n}f_m(S_{k-1}) \overset{d}{\longrightarrow}\int_{-\infty}^\infty \abs{f(x)}\ind{|f(x)|>m} dx \cdot \ell_\alpha(1,0), n\to\infty. 
$$
By the Skorokhod representation theorem, there exist random variables $\ell'_\alpha(1,0)$ and $S_{k-1}^{m,n}$, $m,n\ge 1, k=1,\dots,n$, such that:
\begin{itemize}
\item $\ell'_\alpha(1,0) \overset{d}{=} \ell_\alpha(1,0)$;
\item 
for each $n,m\ge 1$, 
\begin{equation}\label{eq:distreqS_k}
\left(S_{k-1}^{m,n},k=1,\dots,n\right) \overset{d}{=} \left(S_{k-1},k=1,\dots,n\right);
\end{equation}
\item  for each $m\ge 1$, 
$$
c_n\sum_{k=1}^{n}f_m\bigl(S^{m,n}_{k-1}\bigr) \to\int_{-\infty}^\infty \abs{f(x)}\ind{|f(x)|>m} dx \cdot \ell'_\alpha(1,0), n\to\infty,
$$
almost surely.
\end{itemize}
 Then, using \eqref{eq:distreqS_k} and the Fatou lemma, we obtain from \eqref{eq:probab>eps} that
\begin{equation*}\begin{gathered}
\limsup_{m\to \infty}\limsup_{n\to \infty}\prob\bigl(\abs{\zeta_n - \zeta_{n}^m}>\eps \bigr) 
\\
\le \ex{\limsup_{m\to \infty}\limsup_{n\to \infty}\prob\biggl(c_n\sum_{k=1}^{n} \theta_k x_k>\eps\biggr)\Bigg|_{x_k=f_m(S^{m,n}_{k-1}),k=1,\dots,n}} \\
\le \ex{\limsup_{m\to \infty}\limsup_{n\to \infty}\frac{c_n\mu }{\eps}\sum_{k=1}^{n}f_m(S^{m,n}_{k-1})}\\
 = \ex{\limsup_{m\to \infty}\frac{\mu}{\eps}\int_{-\infty}^\infty \abs{f(x)}\ind{|f(x)|>m} dx \cdot \ell'_\alpha(1,0)} = 0.
\end{gathered}\end{equation*}
Therefore, using  \cite[Theorem 3.2]{billingsley}, we get \eqref{eq:discreteconv} also in this case. 

\textit{Step 3}. Taking into account \eqref{eq:ct}, the independence of $\zeta_n$  of $N_t$, and the convergence $N_t\to\infty$, $t\to \infty$, a.s., we get
$$
c_t\sum_{k=1}^{N_t} \theta_k f(S_{k-1}) \overset{d}{\longrightarrow} \mu^{1/\alpha} \int_{-\infty}^\infty f(x)dx\cdot  \ell_\alpha(1,0), t\to\infty. 
$$

It remains to handle the term $c_t\big(t-\tau_{N_t}\big) f\big(S_{N_t}\big)$. Clearly, $\big(t-\tau_{N_t}\big)\le \theta_{N_t + 1}$. Therefore, appealing to \eqref{eq:ct} and to the almost sure convergence $N_t\to \infty$, it suffices to show that 
$c_n \theta_{n+1} f\big(S_{n}\big)\overset{\prob}{\longrightarrow} 0$, $n\to\infty$. Since $\theta_n$ are identically distributed, they are bounded in probability, so we only need to show that $c_n f(S_n) \overset{\prob}{\longrightarrow} 0$, $n\to\infty$. This, however, easily follows from \eqref{eq:absconv}. Indeed, to we clearly have also
\begin{equation}\label{eq:tp}
c_n\sum_{k=1}^{n-1} |f(S_{k-1})| \overset{d}{\longrightarrow} \int_{-\infty}^\infty |f(x)|dx\cdot  \ell_\alpha(t,0), n\to\infty.
\end{equation}
But if  $\limsup_{n\to\infty}\prob\bigl(c_n |f(S_n)|\ge \eta\bigr)$ were positive for some $\eta>0$, the limiting distribution of $c_n\sum_{k=1}^{n} |f(S_{k-1})| = c_n |f(S_n)| + c_n\sum_{k=1}^{n-1} |f(S_{k-1})|$ would strictly dominate that of $c_n\sum_{k=1}^{n-1} |f(S_{k-1})|$, which would contradict \eqref{eq:absconv} and \eqref{eq:tp}. 
\end{proof}

\begin{Lemma}\label{lem:sum1/lambda}
Assume {\rm B1, B2} and let a function $h\colon \R\to [0,\infty)$ satisfy 
\begin{itemize}
\item [{\rm\textbf{H1.}}]
For any $\delta>0$, $\sup_{|x|\le n} h(x) = o(n^\delta), n\to \infty$. 
\item [{\rm\textbf{H2.}}] There exists $\bar h\ge 0$ such that for some $r>\alpha$ there is a uniform convergence of Cezaro averages:
$$
\sup_{|x|\le t^r} \abs{\frac{1}{t} \int_x^{x+t} h(y)dy - \bar h}\to 0, t \to +\infty.
$$
\end{itemize}
Then,
$$
\frac1n\sum_{k=1}^n h(S_k)\overset{\prob}{\longrightarrow} \overline{h}, n\to\infty. 
$$
\end{Lemma}

\begin{proof}
Take some $b\in (1,r/\alpha)$ and define $h_n(x)= h(x)\ind{|x|\le n^b} + \bar h \ind{|x|> n^b}$. 
It follows from H1 that for any $\delta>0$, 
\begin{equation}\label{eq:suphn}
\sup_{x\in \R} |h_n(x)| = o(n^\delta), n\to \infty.
\end{equation}
Clearly, extending $h$ by $\bar h$ may only improve convergence to $\bar h$, so it follows  from H2, that for any sequence $(a_n, n\ge1)$ such that $a_n\ge n^{b/r}$,
\begin{equation}\label{eq:uniformergodic}
\sup_{x\in \R} \abs{\frac{1}{a_n} \int_x^{x+a_n} h_n(y)dy - \bar h}\to 0, n \to \infty.
\end{equation}
Now observe that for any $a\in (0,1-b^{-1})$,
\begin{gather*}
\frac1n\ex{\abs{\sum_{k=1}^n h(S_k) - \sum_{i=1}^n h_n(S_k)}}\le \frac1n \sum_{k=1}^n \ex{\abs{h(S_k)-h_n(S_k)}}\le \frac{C}{n} \sum_{k=1}^n \ex{|S_k|^a\ind{|S_k|\ge n^b}}\\
\le Cn^{b(1-a)-1}  \sum_{k=1}^n \ex{|S_k|} \le  C n^{b(a-1)+1} \ex{|\xi_1|}\to 0, n\to\infty.
\end{gather*}
Therefore, it is enough to prove that 
$$
\frac1n\sum_{k=1}^n h_n(S_k)\overset{\prob}{\longrightarrow} \overline{h}, n\to\infty. 
$$
To this end, consider 
$$
\phi_n^{s,t} = \frac1n\sum_{sn\le k< tn} h_n(S_k) = \sum_{k:sn\le k< tn} F_n\big(X_n(k/n)\big), s<t, n\ge1,
$$
where $X_n(k/n) = n^{-1/\alpha}S_k$, $F_n(x) = n^{-1}h_n(n^{1/\alpha}x)$. As it was proved in \cite{kulik}, the processes $X_n$ provide a Markov approximation for the $\alpha$-stable L\'evy motion $Z$, therefore, we can use \cite[Theorem 1]{kartashov-kulik} about the convergence of additive functionals (concerning the terminology, we advise to consult the articles \cite{kartashov-kulik,kulik}). First note that $$
\sup_{x\in \R} |F_n(x)|\le \frac1n \left(\sup_{|x|\le T_n} |h(x)| + \bar h\right) \to 0, n\to \infty. 
$$
Further, the characteristic of the limiting functional $f^{s,t}(x):= \bar h\cdot  (t-s), s<t$ does not depend on $x$, so obviously satisfies the uniform continuity assumption of  \cite[Theorem 1]{kartashov-kulik}. It then remains to show the uniform (in $x\in \R, 0\le s<t\le 1$) convergence of characteristics
$$
f_n^{s,t}(x) := \sum_{k:sn\le k< tn} \ex{F_n\big(X_n(k/n)+x\big)} = \frac{1}{n}\sum_{k:sn\le k< tn} \ex{h_n(S_k + n^{1/\alpha}x)} 
$$
to $f^{s,t}(x)$. Since $f^{s,t}(x)$ is independent of $x$, this is equivalent to the uniform convergence of $f_n^{s,t}(n^{-1/\alpha}x)$. 

Fix some $\varepsilon\in (b\alpha/r,1)$ and consider  
\begin{gather*}
\ex{h_n(S_k + x)} = \int_{-\infty}^\infty h_n(k^{1/\alpha}y+x) f_{k^{-1/\alpha}S_k}(y)dy, k\ge n^{\eps}. 
\end{gather*}
By \cite{banys} (see also \cite{basu}), 
$$
\int_{-\infty}^{\infty} |f_{Z_\alpha}(y) - f_{k^{-1/\alpha}S_k}(y)|dy = o(k^{1-2/\alpha}), k\to \infty. 
$$
Therefore,  for  any $\delta\in \bigl(0,\eps(2/\alpha-1)\bigr)$, thanks to \eqref{eq:suphn},
\begin{equation}\label{eq:ex(hn)-ex(hn)}
\begin{gathered}
\sup_{x\in \R, k\ge n^\eps}\big|\ex{h_n(S_k + x)} - \ex{h_n(k^{1/\alpha }Z_\alpha+x)} \big|\\
\le \sup|h_n|\cdot  \sup_{k\ge n^\eps} \int_{-\infty}^{\infty} |f_{Z_\alpha}(y) - f_{k^{-1/\alpha}}(y)|dy = o\big(n^{\delta +\eps(1 - 2/\alpha)}\big)\to 0, n\to\infty.
\end{gathered}
\end{equation}
Further,
\begin{gather*}
\ex{h_n(k^{1/\alpha }Z_\alpha+x)} =  \int_{-\infty}^\infty h_n(k^{1/\alpha}y+x) f_{Z_\alpha}(y)dy\\
 =  \int_0^\infty \int_{y: f_{Z_\alpha}(y)\ge z} h_n(k^{1/\alpha} y + x)dy\, dz.
\end{gather*}
It is well known (see e.g. \cite[Chapter 2]{zolotarev}) that a stable distribution has an unimodal analytic density, so for each $z\in [0,\max f_{Z_\alpha})$, there exist some $a_z<b_z$ such that $\set{y: f_{Z_\alpha}(y)\ge z} =[a_z,b_z]$. Then we can write for some $\gamma\in \bigl(0,2(b/r-\eps/\alpha)\bigr)$
\begin{equation}\label{eq:ex(hn)}
\begin{gathered}
\ex{h_n(k^{1/\alpha }Z_\alpha+x)} = \int_0^{\max f_{Z_\alpha}} \int_{a_z}^{b_z} h_n(k^{1/\alpha }y + x)dy\, dz\\ = \left(\int_0^{\max f_{Z_\alpha}-n^{-\gamma}} + \int_{\max f_{Z_\alpha}-n^{-\gamma}}^{\max f_{Z_\alpha}} \right)\int_{a_z}^{b_z} h_n(k^{1/\alpha }y + x)dy\, dz. 
\end{gathered}
\end{equation}
Clearly, 
$$
\abs{\int_{a_z}^{b_z} h_n(k^{1/\alpha }y + x)dy} \le  \sup|h_n| (b_z-a_z) n^{-\gamma},
$$
whence, in view of \eqref{eq:suphn}, 
\begin{equation}\label{eq:inttail}
\sup_{x\in\R, k\ge 1}\abs{\int_{\max f_{Z_\alpha}-n^{-\gamma}}^{\max f_{Z_\alpha}} \int_{a_z}^{b_z} h_n(k^{1/\alpha }y + x)dy \, dz} \to 0, n\to\infty. 
\end{equation}
Further, for $z\le \max f_{Z_\alpha} - n^{-\gamma}$,
$$
\int_{a_z}^{b_z} h_n(k^{1/\alpha }y + x)dy = \frac{1}{k^{1/\alpha}}\int_{k^{1/\alpha} a_z + x}^{k^{1/\alpha} b_z + x} h(u)du = \frac{b_z-a_z}{k^{1/\alpha}(b_z-a_z)}\int_{k^{1/\alpha} a_z + x}^{k^{1/\alpha} b_z + x} h(u)du.
$$
Thanks to continuous differentiability of $f_{Z_\alpha}$, there exists some positive $c>0$ such that $b_z - a_z \ge c n^{-\gamma/2}$ for any $z\le \max f_{Z_\alpha} - n^{-\gamma}$. Therefore, for such $z$ and for $k\ge n^\eps$, $k^{1/\alpha} (b_z-a_z)\ge c n^{\eps/\alpha -\gamma/2} \ge n^{b/r}$ for all $n$ large enough. Consequently, in view of \eqref{eq:uniformergodic},
$$
\sup_{\substack{x\in \R, k\ge n^\eps\\z\le \max f_{Z_\alpha} - n^{-\gamma} }}\abs{\frac{1}{k^{1/\alpha}(b_z-a_z)}\int_{k^{1/\alpha} a_z + x}^{k^{1/\alpha} b_z + x} h(u)du - \bar h} \to 0, n\to \infty. 
$$
Combining this with \eqref{eq:ex(hn)}--\eqref{eq:inttail} and noting that $\int_0^{\max f_{Z_\alpha}} (b_z - a_z)dz = \int_{-\infty}^\infty f_{Z_\alpha}(x)dx = 1$, we get 
\begin{gather*}
\limsup_{n\to \infty }\sup_{x\in \R, k\ge n^\eps}\abs{\ex{h_n(k^{1/\alpha }Z_\alpha+x)} - \bar h}\\
 = 
\limsup_{n\to \infty }\sup_{x\in \R, k\ge n^\eps}\abs{\int_{0}^{\max f_{Z_\alpha}} \int_{a_z}^{b_z} h_n(k^{1/\alpha }y + x)dy\, dz- \int_{0}^{\max f_{Z_\alpha}} (b_z-a_z)\bar h \,dz}\\
=
\limsup_{n\to \infty }\sup_{x\in \R, k\ge n^\eps}\abs{\int_0^{\max f_{Z_\alpha}-n^{-\gamma}} (b_z-a_z)\left(\frac{1}{k^{1/\alpha}(b_z-a_z)}\int_{k^{1/\alpha} a_z + x}^{k^{1/\alpha} b_z + x} h(u)du  - \bar h\right) dz}\\
\le \limsup_{n\to \infty }\int_0^{\max f_{Z_\alpha}} (b_z - a_z)dz\cdot \sup_{\substack{x\in \R, k\ge n^\eps\\z\le \max f_{Z_\alpha} - n^{-\gamma} }}\abs{\frac{1}{k^{1/\alpha}(b_z-a_z)}\int_{k^{1/\alpha} a_z + x}^{k^{1/\alpha} b_z + x} h(u)du - \bar h} =0.
\end{gather*}
Recalling \eqref{eq:ex(hn)-ex(hn)}, we arrive at
\begin{gather*}
\sup_{x\in \R, k\ge n^\eps}\big|\ex{h_n(S_k + x)} - \bar h \big|\to 0, n\to\infty,
\end{gather*}
whence 
$$
\sup_{x\in \R,  n^{\eps-1}\le s<t\le 1} \abs{\frac1n \sum_{k:sn\le k< tn} \ex{h_n(S_k + x)} - \bar h\cdot (t-s)}\to 0, n\to \infty. 
$$
Also, thanks to \eqref{eq:suphn},
$$
\sup_{x\in \R, s\le n^{\eps-1}}\abs{\frac{1}{n}\sum_{k< ns} \ex{h_n(S_k + x)}}\le Cn^{\eps-1}\sup_{x\in \R}|h_n(x)| \to 0, n\to\infty. 
$$
Consequently,
\begin{gather*}
\sup_{x\in \R,  0\le s<t\le 1} \abs{\frac1n \sum_{k:sn\le k< tn} \ex{h_n(S_k + x)} - \bar h\cdot (t-s)}\to 0, n\to \infty. 
\end{gather*}
This shows the required uniform convergence of characteristics, so by \cite[Theorem 1]{kartashov-kulik} we get
$$
\frac1n\sum_{k=1}^n h_n(S_k)\longrightarrow \overline{h}, n\to\infty,
$$
in law, equivalently, 
in probability. 
\end{proof}

\begin{proof}[Proof of Proposition~\ref{prop:N_t}]
Denote $\gamma_n = \frac{1}{n}\sum_{i=1}^n \frac{\theta_i}{\Lambda(S_i)}$ and
write, similarly to the proof of Theorem~\ref{thm:main}, for any $\lambda \in \R\setminus\set{0}$,
\begin{equation*}\begin{gathered}
\varphi_{\gamma_n}(\lambda) = \ex{\ex{\set{\frac{i}{n} \sum_{k=1}^n\theta_k x_k}}\Bigg|_{x_k=\Lambda(S_{k-1})^{-1} ,k=1,\dots,n}} = \ex{\prod_{k=1}^n\exp\set{ \varphi_{\theta_1}\Bigl(\frac{\lambda}{n\Lambda(S_k)}\Bigr)}}\\
= \ex{\exp\set{\sum_{k=1}^n \Log\varphi_{\theta_1}\Bigl(\frac{\lambda}{n\Lambda(S_k)}\Bigr)}}
=  \ex{\exp\set{\frac{i\mu \lambda}{n}\sum_{k=1}^n   \frac1{\Lambda(S_{k-1})} + R_{n}}},
\end{gathered}\end{equation*}
where
\begin{gather*}
R_{n} = \sum_{k=1}^n r\Bigl(\frac{\lambda}{n\Lambda(S_k)}\Bigr),\
r(x) = \Log \varphi_{\theta_1}(x) - i\mu x = o(x), x\to 0. 
\end{gather*}
By Lemma~\ref{lem:sum1/lambda}, 
\begin{equation}\label{eq:Y_n->0}
Y_n:= \frac{1}{n}\sum_{k=1}^n   \frac1{\Lambda(S_{k-1})} \overset{\prob}{\longrightarrow} \overline{\Lambda^{-1}}, n\to\infty. 
\end{equation}
In order to prove the first claim it remains to show that $R_n\overset{\prob}{\longrightarrow} 0$, $n\to\infty$. Fix some $\eps>0$. For any $a>0$, there exists some $\delta>0$ such that $\abs{r(x)}\le a |x|$ for $|x|<\delta$. Therefore, on the event $A_n:=\set{\max_{k\le n}\Lambda(S_k)^{-1}\le n\delta/\abs{\lambda} }$, we have $\abs{R_n}\le a Y_n$. Therefore, 
$$
\prob\left(|R_n|>\eps\right)\le \prob(Y_n > \eps/a) + \prob(A_n^c).
$$
Choosing $a<\eps/\overline{\Lambda^{-1}}$, we get from \eqref{eq:Y_n->0} that $\prob(Y_n > \eps/a)\to 0$, $n\to\infty$. On the other hand, since by B3 for any $\eta<1$ it holds that $\Lambda(x)^{-1}\le K_\eta|x|^\eta$ with some $K_\eta>0$, we have
\begin{gather*}
\prob(A_n^c) \le \sum_{k=1}^n \prob\biggl(\Lambda(S_k)^{-1} \ge \frac{n\delta}{|\lambda|} \biggr)\le \sum_{k=1}^n \prob\biggl(|S_k|^\eta \ge \frac{n\delta}{K_\eta\abs{\lambda}} \biggr)\le \sum_{k=1}^n \prob(|S_k|\ge  Cn^{1/\eta} \biggr)\\ 
\le n^2 \prob(\abs{\xi_1}\ge C n^{1/\eta-1}) = n^2 O(n^{\alpha(1-1/\eta)}), n\to\infty,
\end{gather*}
where the last follows from B1 (see e.g.\ \cite[Section 1.1]{Borodin}). Taking $\eta <(1+2/\alpha)^{-1}$, we get $\prob(A_n^c)\to 0$, $n\to\infty$, thus establishing the convergence $R_n\overset{\prob}{\longrightarrow} 0$, $n\to\infty$, which finishes the proof for the first claim that $\gamma_n \mu  {\prob}{\longrightarrow} \overline{\Lambda^{-1}}$, $n\to\infty$.  The second one follows in a standard way: for any $x<\bigl(\mu \overline{\Lambda^{-1}}\bigr)^{-1}$,
$$
\prob\left(N_t \le tx\right)  = \prob\left(\sum_{i=1}^{[tx]} \frac{\theta_i}{\Lambda(S_i)}\ge t\right) = \prob\Bigl(\gamma_{[tx]}\ge \frac{t}{[tx]} \Bigr)\to 0, t\to +\infty,
$$
since $\lim_{t\to \infty}\frac{t}{[tx]} = \frac1x <\mu \overline{\Lambda^{-1}}$, and similarly for any $x>\bigl(\mu \overline{\Lambda^{-1}}\bigr)^{-1}$, $\prob\left(N_t \ge tx\right)\to 0$, $t\to \infty$. 
\end{proof} 
\begin{proof}[Proof of Theorem~\ref{thm:main-lambda}]
Similarly to \eqref{decompos}, we can write
\begin{equation*}
\int_0^t g(X_s)ds = \sum_{k=1}^{N_t} \theta_k \frac{g(S_{k-1})}{\Lambda(S_{k-1})} + g\big(t-\tau_{N_t}\big) \frac{g(S_{N_t})}{\Lambda(S_{N_t})}.
\end{equation*}
From Proposition~\ref{prop:N_t}, we have $N_t/t \overset{\prob}{\longrightarrow} \bigl(\mu \cdot \overline{\Lambda^{-1}}\bigr)^{-1}$, $n\to\infty$. Therefore, repeating the proof of Theorem~\ref{thm:main}, we arrive at the statement.
\end{proof}
The following lemma is probably well known, but we include it for completeness.
\begin{Lemma}\label{lem:aindependent}
Let $\set{Y_t,t\in [0,T]}$ be a centered measurable process which is $a$-independent for some $a\in (0,T)$, i.e. $\{Y_t,t\in A\}$ and $\set{Y_t,t\in B}$ are independent whenever $\inf_{t\in A,s\in B}|t-s|\ge a$. For each integer $k\ge 1$, there exists a universal constant $C_k>0$ such that 
$$
\ex{\left(\int_0^T Y_t dt\right)^{2k}}\le C_k (aT)^k \sup_{t\in[0,T]} \ex{Y_t^{2k}}.
$$
\end{Lemma}
\begin{proof}
Since $Y$ is centered and $a$-independent, we have
\begin{gather*}
\ex{\left(\int_0^T Y_t dt\right)^{2k}} = \int_{S_{2k,a,T}} \ex{\prod_{i=1}^{2k} Y_{t_i}}dt_1\dots dt_{2k},
\end{gather*}
where $S_{2k,a,T} = \set{(t_1,\dots,t_{2k})\in [0,T]\mid \forall i=1,\dots,2k\ \exists j\neq i: |t_i-t_j|\le a}$. Using the H\"older inequality, we get
\begin{equation}\label{eq:intermineq}
\ex{\left(\int_0^T Y_t dt\right)^{2k}}\le \lambda(S_{2k},a,T)\sup_{t\in[0,T]} \ex{Y_t^{2k}}. 
\end{equation}
Clearly, $\lambda(S_{2k},a,T) = T^{2k} \lambda(S_{2k},a/T,1)$. In turn, 
$$\lambda(S_{2k},a/T,1) = \prob\left(\forall i=1,\dots,2k\ \exists j\neq i: |U_i - U_j|\le a/T\right),$$
where $U_1,\dots,U_{2k}$ are iid $U(0,1)$ random variables. Denote by $\mathcal G_{2k}$ the set of all graphs on $N_{2k}:=\set{1,\dots,2k}$ having no isolated vertices; for $G\in \mathcal G_{2k}$, let $V(G)$ be its set of edges, and $S(G)$ be its minimal vertex cover, i.e. the minimal (in cardinality) set of vertices adjacent to all edges of $G$. It is well known that $|S(G)|$ is equal to the number of edges in the maximal matching (disjoint set of edges) of $G$, so $S(G)\le k$.  Then 
\begin{gather*}
\prob\left(\forall i=1,\dots,2k\ \exists j\neq i: |U_i - U_j|\le a/T\right) = \prob\Biggl(\bigcup_{G\in \mathcal G_{2k}}\bigcap_{{i,j}\in V(G)} \{|U_i - U_j|\le a/T\}\Biggr) \\
\le \sum_{G\in \mathcal G_{2k}} \prob\Biggl(\bigcap_{i\in N_{2k}\setminus S(G)}\bigcup_{j\in S(G)} \{|U_i - U_j|\le a/T\}\Biggr) \\
= \sum_{G\in \mathcal G_{2k}} \ex{\prob\Biggl(\bigcap_{i\in N_{2k}\setminus S(G)}\bigcup_{j\in S(G)} \{|U_i - x_j|\le a/T\}\Biggr)\Bigg|_{x_j = U_j, j\in S(G)}}\\
= \sum_{G\in \mathcal G_{2k}} \ex{\prod_{i\in N_{2k}\setminus S(G)}\prob\Biggl(\bigcup_{j\in S(G)} \{|U_i - x_j|\le a/T\}\Biggr)\Bigg|_{x_j = U_j, j\in S(G)}}\\
\le \sum_{G\in \mathcal G_{2k}} \ex{\prod_{i\in N_{2k}\setminus S(G)}\sum_{j\in S(G)}\prob\Biggl( \{|U_i - x_j|\le a/T\}\Biggr)\Bigg|_{x_j = U_j, j\in S(G)}}\\
\le \sum_{G\in \mathcal G_{2k}} \ex{\prod_{i\in N_{2k}\setminus S(G)}\left(|S(G)|\cdot \frac{2a}{T}\right) }\\
\le \sum_{G\in \mathcal G_{2k}} \Bigl(\frac{2ka}{T}\Bigr)^{2k-|S(G)|}\le \Bigl(\frac{a}{T}\Bigr)^k\sum_{G\in \mathcal G_{2k}} (2k)^{2k-|S(G)|} = C_k\Bigl(\frac{a}{T}\Bigr)^k.
\end{gather*}
Recalling the fact that $\lambda(S_{2k},a,T)$ is $T^{2k}$ times this expression and the estimate \eqref{eq:intermineq}, we arrive at the statement.
\end{proof}
\begin{proof}[Proof of Proposition~\ref{prop:b3b4}]
The first statement follows immediately from \eqref{eq:E-phi}. 
In order to establish the second one, we start by noting that, in view of \eqref{eq:B3}, for large $t$ the average of $X$ over $[x,x+t]$ will be close to that over $[x,x+\lfloor t\rfloor]$, where $\lfloor t\rfloor$ is the integer part of $t$, so it is enough to show the convergence over integers. 
Most of the statements below will hold almost surely, so for brevity, we omit this phrase throughout. 

Fix some $a\in (0,1)$ define $
\phi_n(x) = \phi(x)\ind{|x|\le n^a}$, $\Lambda_n(x,\gamma) = e^{-E_{\phi_n}(x,\gamma)}$, $\bar \phi_n = \phi - \phi_n$.
Let $\nu_k = \abs{\gamma\cap[k-\frac12,k+\frac12]}$, $k\in \mathbb Z$. It is easy to show (see e.g. \cite[Lemma 2.1]{carmona}) that
$$
\sup_{k\in \mathbb Z} \frac{\nu_k}{l(k)/l(l(k))} <\infty,
$$
where $l(x) = 2 + \log(2+|x|), x\in \R$.

Therefore, for any $x\in \R$ and any $\eta\in(0,a\beta)$, using C1, we have
\begin{gather*}
\abs{E_{\phi}(x,\gamma) - E_{\phi_n}(x,\gamma)} = \abs{E_{\bar \phi_n}(x,\gamma)} \le \sum_{y\in \gamma} \abs{\bar\phi_n(x-y)} \\
\le C\sum_{k\in \mathbb Z, |k-x|\ge n^a - 1} \frac{\nu_k}{1 + |k-x|^{\beta+1}} \le 
C \sum_{m\in \mathbb Z, |m|\ge n^a - 1} \frac{1}{1 + m^{\beta+1}}\cdot \frac{l(m+x)}{l\bigl(l(m+x)\bigr)} 
\\ \le  C \sum_{|m|\in \mathbb Z, |m|\ge n^a - 1} \frac{1}{1 + m^{\beta+1}}\cdot \left(\frac{l(m)}{l(l(m))}+\frac{l(x)}{l(l(x))} \right) \le
\\
\le C \left(n^{-a\beta+\eta} + n^{-a\beta}\cdot\frac{l(x)}{l(l(x))} \right) \le  C n^{-a\beta+\eta}\cdot \frac{l(x)}{l(l(x))}.
\end{gather*}
Hence, owing to \eqref{eq:E-phi}, we get that for any $r>1$,
$$
\sup_{|x|\le 2 n^r} \abs{\Lambda(x,\gamma)^{-1} - \Lambda_n(x,\gamma)^{-1}} \to 0, n\to\infty,
$$
consequently, 
$$
\sup_{|x|\le n^r}\abs{\frac{1}{n} \int_x^{x+n} \Lambda(y,\gamma)^{-1}dy - \frac{1}{n} \int_x^{x+n} \Lambda_n(y,\gamma)^{-1}dy}\to 0, n\to\infty. 
$$
Since $\Lambda_n(0,\gamma)\le \Lambda(0,\gamma)$ and $\Lambda_n(0,\gamma)\to \Lambda(0,\gamma)$, $n\to\infty$, then $\ex{\Lambda_n(0,\gamma)^{-1}}\to\ex{\Lambda(0,\gamma)^{-1}}$, $n\to\infty$, so we are left to show that 
$$
\sup_{|x|\le n^r}\abs{\frac{1}{n} \int_x^{x+n} \Lambda_n(y,\gamma)^{-1}dy - \ex{\Lambda_n(0,\gamma)^{-1}}}\to 0, n\to\infty.
$$
Observe that the process $\Lambda_n(y,\gamma)$ is $2n^{a}$-independent. Then, using the stationarity of $\Lambda_n$, we obtain from Lemma~\ref{lem:aindependent} that  for any $k\ge 1$, 
\begin{gather*}
\ex{\left(\frac1n \int_x^{x+n} \Lambda_n(y,\gamma)^{-1}dy - \ex{\Lambda_n(0,\gamma)^{-1}}\right)^{2k}}\\ 
= n^{-k}\ex{\left( \int_x^{x+n} \bigl(\Lambda_n(y,\gamma)^{-1}-  \ex{\Lambda_n(y,\gamma)^{-1}}\bigr)dy \right)^{2k}}
\\
\le C_k n^{k(a-1)} \ex{\left(\Lambda_n(0,\gamma)^{-1} - \ex{\Lambda_n(0,\gamma)^{-1}}\right)^{2k}}\\
\le C_k \ex{\left(\Lambda(0,\gamma)^{-1} + \ex{\Lambda(0,\gamma)^{-1}}\right)^{2k}}  n^{k(a-1)}.
\end{gather*}
By Markov's inequality, for any $\eps>0$,
\begin{gather*}
\prob\left(\abs{\frac1n \int_x^{x+n} \Lambda_n(y,\gamma)^{-1}dy - \ex{\Lambda_n(0,\gamma)^{-1}}}\ge \eps\right)\le \frac{C_k n^{k(a-1)}}{\eps^k}.
\end{gather*}

Define the set $A_n = \set{ n^{r-a} i, i= - [n^a], \dots, [n^a]+1}$ and for $x\in [-n^r,n^r]$ denote $a_n(x) = \sup\{y\in A_n, y\le x\}$. Thanks to \eqref{eq:B3},
\begin{gather*}
\sup_{|x|\le n^r}\abs{\frac1n \int_x^{x+n} \Lambda_n(y,\gamma)^{-1}dy -\frac1n \int_{a_n(x)}^{a_n(x)+n} \Lambda_n(y,\gamma)^{-1}dy} \\
\le \sup_{|x|\le n^r}\frac{2(x-a_n(x))}{n}\cdot \sup_{|y|\le 2 n^r} \Lambda_n(y)^{-1} \le  2n^{a-1}\sup_{|y|\le 2 n^r} \Lambda(y)^{-1}\to 0, n\to\infty. 
\end{gather*}
Consequently, 
\begin{gather*}
\limsup_{n\to\infty}\prob\left(\sup_{|x|\le n^r} \abs{\frac1n \int_x^{x+n} \Lambda_n(y,\gamma)^{-1}dy - \ex{\Lambda_n(0,\gamma)^{-1}}}\ge \eps \right) \\
= \limsup_{n\to\infty}\prob\left(\sup_{x\in A_n} \abs{\frac1n \int_x^{x+n} \Lambda_n(y,\gamma)^{-1}dy - \ex{\Lambda_n(0,\gamma)^{-1}}}\ge \eps \right)\\
\le \limsup_{n\to\infty}\sum _{x\in A_n}\prob\left(\abs{\frac1n \int_x^{x+n} \Lambda_n(y,\gamma)^{-1}dy - \ex{\Lambda_n(0,\gamma)^{-1}}}\ge \eps \right)\\
\le \limsup_{n\to\infty}\sum _{x\in A_n} \frac{C_k n^{k(a-1)}}{\eps^k} \le 
 C_{k,\eps}\limsup_{n\to\infty} n^{k(a-1) + a}. 
\end{gather*}
Now taking $k> (1+a)/(1-a)$, we obtain that
$$
\sup_{|x|\le n^r} \abs{\frac1n \int_x^{x+n} \Lambda_n(y,\gamma)^{-1}dy - \ex{\Lambda_n(0,\gamma)^{-1}}}\to 0, n\to\infty,
$$
by virtue of the Borel--Cantelli lemma, concluding the proof. \end{proof}

\begin{proof}[Proof of Theorem~\ref{thm:main-poisson}]
Since $\gamma$ is independent of $X$, it suffices to show the quenched weak convergence, i.e.\ that 
the required weak convergence holds for almost all fixed realizations of $\gamma$. This, in turn, boils down to verifying  the assumptions B3, B4 for $\Lambda$ and A2 for $f = g/\Lambda$. The former follow from Proposition~\ref{prop:b3b4}. Concerning the latter,  note that 
\begin{gather*}
\ex{\norm{\frac{g}\Lambda}_{L^1(\R)}} = \ex{\int_{-\infty}^\infty\frac{\abs{g(x)}}{\Lambda(x,\gamma)} dx } = \int_{-\infty}^\infty \abs{g(x)}\ex{e^{E_\phi(x,\gamma)}}dx\\
= \int_{-\infty}^\infty \abs{g(x)}\ex{e^{E_\phi(x,\gamma)}}dx = \int_{-\infty}^\infty \abs{g(x)}dx\cdot  \exp\set{\int_{-\infty}^\infty \bigl(e^{\phi(y)}-1\bigr)dy}<\infty.
\end{gather*}
Consequently, $g/\Lambda \in L^1(\R)$ a.s.
Similarly, if $g\in L^2(\R)$, then
\begin{gather*}
\ex{\norm{\frac{g}\Lambda}^2_{L^2(\R)}} = \int_{-\infty}^\infty g(x)^2\ex{e^{2E_\phi(x,\gamma)}}dx =  \int_{-\infty}^\infty g(x)^2 dx\cdot  \exp\set{\int_{-\infty}^\infty \bigl(e^{2\phi(y)}-1\bigr)dy}<\infty
\end{gather*}
and $g/\Lambda\in L^2(\R)$ a.s.; if $|g(x)|\le C(1+|x|^\eps)^{-1}$, then $g/\Lambda$ is bounded thanks to B3. Consequently, B3, B4, and A2 hold for almost all $\gamma$, which implies the required quenched weak convergence.
\end{proof}
\end{document}